\documentclass[12pt,letterpaper]{article}

\usepackage{amssymb,amsfonts,amscd,amsthm}
\usepackage{thmtools}
\usepackage[all,arc]{xy}
\usepackage{enumerate, bbm}
\usepackage{mathrsfs}
\usepackage{tikz}
\usepackage[left=0.8in,top=0.8in,right=0.8in,bottom=0.8in]{geometry}
\usepackage{mathtools}
\usepackage{hyperref}
\usepackage{cleveref}
\usepackage{color}
\usepackage{graphicx}
\usepackage{enumitem} 
\graphicspath{ {TexImages/} }

\newtheorem{thm}{Theorem}[section]
\newtheorem{cor}[thm]{Corollary}
\newtheorem{prop}[thm]{Proposition}
\newtheorem{lem}[thm]{Lemma}
\newtheorem{claim}[thm]{Claim}

\newtheorem{quest}[thm]{Question}
\newtheorem{prob}[thm]{Problem}

\theoremstyle{definition}
\newtheorem{defn}{Definition}

\hypersetup{
	colorlinks,
	citecolor=blue,
	filecolor=blue,
	linkcolor=blue,
	urlcolor=blue,
	linktocpage
}

\setlist[enumerate]{itemsep=2ex, topsep=2ex} 
\setlist[itemize]{itemsep=2ex, topsep=2ex}

\newcommand{\N}{\mathbb{N}}

\newcommand{\E}{\mathbb{E}}
\newcommand{\al}{\alpha}
\newcommand{\be}{\beta}

\newcommand{\ep}{\varepsilon}

\newcommand{\Om}{\Omega}

\newcommand{\del}{\delta}
\newcommand{\Del}{\Delta}

\newcommand{\half}{\frac{1}{2}}

\newcommand{\sm}{\setminus}
\newcommand{\sub}{\subseteq}

\renewcommand{\c}[1]{\mathcal{#1}}

\newcommand{\mr}[1]{\mathrm{#1}}

\renewcommand{\SS}[1]{\textcolor{red}{#1}}

\newcommand{\ex}{\mr{ex}}

\newcommand{\Var}{\mr{Var}}

\newcommand{\num}{\#}
\newcommand{\planar}{P}

\title{Supersaturation for Hypergraph-Weighted Independent Sets}
\author{Sam Spiro\footnote{Dept.\ of Mathematics and Statistics, Georgia State University, sspiro@gsu.edu}.}
\date{\today}

\begin{document}
	\maketitle
\begin{abstract}
	Many extremal problems can be viewed as finding large independent sets in an auxiliary hypergraph.  We propose a generalization of this by looking for ``large'' independent sets $I$ in a hypergraph $\mathcal{F}$ where ``large'' is measured by how many edges $I$ induces in another hypergraph $\mathcal{H}$ on the same vertex set as $\mathcal{F}$.  We prove general supersaturation results for such extremal problems motivated by the breakthrough work of Ferber, McKinley and Samotij on counting $F$-free graphs.  As applications, we prove new supersaturation bounds for generalized Tur\'an problems, as well as supersaturation bounds for a new set of extremal problems inspired by work of Fox and Pohoata on finding subsets $A\sub\mathbb{N}$ which maximize the number of solutions to a given system of equations while avoiding solutions to another system.
\end{abstract}

\section{Introduction}
Classical problems in extremal combinatorics ask how large an object can be before it contains a given structure $X$.  Supersaturation problems refine this by asking how many copies of $X$ are guaranteed to exist in an object of a given size.  

Supersaturation results are interesting in their own right as well as useful tools for other problems.  For example, the best known upper bounds for Tur\'an numbers of complete bipartite graphs \cite{kovari1954problem} and the hypercube \cite{erdos1984cube} use supersaturation results. Supersaturation results have become increasingly relevant due to them being key ingredients in applications of the powerful method of hypergraph containers due independently to Balogh, Morris, and Samotij~\cite{balogh2015independent} and Saxton and Thomason~\cite{saxton2015hypergraph}.  For example, the breakthrough work of Ferber, McKinley, and Samotij~\cite{ferber2020supersaturated} on counting $F$-free graphs uses at its core a supersaturation result for Tur\'an problems together with an application of the hypergraph container method.

Partially inspired by this work of Ferber, McKinley, and Samotij, we establish general supersaturation results together with an abstract framework for viewing extremal problems as finding large weighted independent sets in hypergraphs.  To give context for our abstract statements, we begin with some consequences of our theorems to the concrete setting of generalized Tur\'an problem supersaturation, as well as to a new set of extremal problem for systems of equations.

\subsection{Generalized Tur\'an Problems}
Given graphs $H$ and $F$, the \textit{generalized Tur\'an number} $\ex(n,H,F)$ is the maximum number of copies of $H$ in an $n$-vertex $F$-free graph, and when $H=K_2$ we write $\ex(n,F):=\ex(n,K_2,F)$ and call this the (classical) \textit{Tur\'an number}.  The definition and systematic study of $\ex(n,H,F)$ was introduced in foundational work by Alon and Shikhelman~\cite{AS2016}, and since then there has seen a tremendous amount of work on the topic \cite{GGMV2020, GP2022, MYZ2018, ZGHLSX2023}; see the recent survey of Gerbner and Palmer~\cite{GP2025} for a more complete look at this area.

The supersaturation problem for generalized Tur\'an problems is significantly harder than that of classical Tur\'an problems.  For example, it is straightforward to show that if $G$ is an $n$-vertex graph with $e(G)\ge k\cdot \ex(n,F)$ then $G$ contains at least $(k-1) \ex(n,F)$ copies of $F$ simply by iteratively deleting edges from $G$ which lie in a copy of $F$. Such a basic result does not hold, however, for arbitrary generalized Tur\'an problems, as can be seen by considering any example with $\ex(n,H,F)\gg n^{v(F)}$ since no graph can have more than $n^{v(F)}$ copies of $F$.  Because of this and other difficulties there has been little work on supersaturation for generalized Tur\'an problems, though see for example \cite{cutler2022supersaturation,dubroff2025clique,gerbner2022unified,halfpap2021supersaturation} for some sporadic results.

Overcoming the obstacles mentioned above, we give the first ever set of lower bounds for generalized Tur\'an problem supersaturation which hold for essentially arbitrary $H$ and $F$.  Here we write $\num(F,G)$ to denote the number of subgraphs of $G$ which are isomorphic to $F$.
\begin{thm}\label{generalized Turan supersaturation}
	Let $H$ be an $h$-vertex graph, $F$ an $f$-vertex graph without isolated vertices, and $\be<h$ such that $\ex(n,H,\c{F})= O(n^\be)$.  There exists some $k_0$ such that the following holds for any $n$-vertex graph $G$ with $\num(H,G)\ge k n^\be$ and $k\ge k_0$:
	\begin{itemize}
		\item[(i)] We have 
		\[\num(F,G)=\Omega\left( k^{\frac{f+h-2-\be}{h-\be}}n^{\be-h+2}\right).\]
		\item[(ii)] If $f\ge h-\be$, then
		\[\num(F,G)=\Omega\left(k^{\frac{f}{h-\be}}\right).\]
		\item[(iii)] If $f<h-\be$ and if $F$ is not isomorphic to a subgraph of $H$, then
		\[\num(F,G)=\Omega\left(k^{\frac{2}{h-\be-f+2}}\right).\]
	\end{itemize}
\end{thm}
The bound (i) when $H=K_2$ recovers a supersaturation result of Jiang and Longbrake~\cite{jiang2022balanced} which is the best known supersaturation bound for classical Tur\'an numbers for arbitrary $F$.  The bounds (i) and (ii) are both best possible when $k=\Theta(n^{h-\be})$, and which of these two bound does better for other values of $k$ depends on how $\be$ compares to $h-2$.  The bound (ii) is tight for all values of $k$ whenever\footnote{For example, if $H=K_3$ and $F=K_{2}\cup K_2$ then one can check $\ex(n,H,F)=1=O(n^0)$ for $n\ge 3$, and since $v(F)>v(H)$ and $F$ has no isolated vertices, \Cref{generalized Turan supersaturation}(ii) gives optimal supersaturation bounds in this case.} $\be=0$ by considering $G$ to be a clique on $k^{1/h}$ vertices.  Bounds (ii) and (iii) combine to give the following intuitive but not so easy to prove fact.

\begin{cor}\label{large number of H implies large number of F}
	If $H,F$ are graphs such that $F$ does not have isolated vertices and is not isomorphic to a subgraph of $H$, then for all $C>0$ there exists some $C'>0$ such that if $G$ is an $n$-vertex graph with $\num(H,G)\ge C' \ex(n,H,F)$ then $\num(F,G)\ge C$.
\end{cor}
That is, any graph with much more than $\ex(n,H,F)$ copies of $H$ contains many copies of $F$ whenever $F$ is not isomorphic to a subgraph of $H$.  The assumption $F$ is not isomorphic to a subgraph of $H$ is necessary in general for this to hold: if $F$ is any graph without isolated vertices and $H=F\cup K_1$, then the graph $G$ which consists of $F$ together with $k$ isolated vertices contains only 1 copy of $F$ despite containing $k\ge k \ex(n,H,F)=0$ copies of $H$.

Because \Cref{generalized Turan supersaturation} is proven with our general supersaturation machinery, we can without effort obtain similar results for related settings.  For example, \textit{the generalized planar Tur\'an number} $\ex_{\planar}(n,H,F)$ is the maximum number of copies of $H$ in an $n$-vertex planar $F$-free graph and is a well-studied variant of the generalized Tur\'an problem \cite{grzesik2022maximum,gyori2021generalized,lv2024maximum}.  Through our general supersaturation results, we can state \Cref{generalized Turan supersaturation} word for word with $\ex(n,H,F)$ replaced throughout by $\ex_{\planar}(n,H,F)$.  We formally record this for the bounds (ii) and (iii) which, as we note in the concluding remarks, are the most interesting cases of \Cref{generalized Turan supersaturation} in the planar setting.

\begin{thm}\label{generalized planar Turan}

	Let $H$ be an $h$-vertex graph, $F$ an $f$-vertex graph without isolated vertices, and $\be<h$ such that $\ex_{\planar}(n,H,\c{F})= O(n^\be)$.  There exists some $k_0$ such that the following holds for any $n$-vertex planar graph $G$ with $\num(H,G)\ge k n^\be$ and $k\ge k_0$: if $f\ge h-\be$,  then\[\num(F,G)=\Omega\left(k^{\frac{f}{h-\be}}\right),\]
	and if $f<h-\be$ and $F$ is not isomorphic to a subgraph of $H$, then
	\[\num(F,G)=\Omega\left(k^{\frac{2}{h-\be-f+2}}\right).\]
\end{thm}
In particular, this implies any planar graph $G$ with at least $n^\ep\cdot \ex_{\planar}(n,H,F)$ copies of $H$ for some $\ep>0$ contains at least $n^\del$ copies of $F$ for some $\del>0$ a function of $H,F,\ep$ provided $F$ has no isolated vertices and is not isomorphic to a subgraph of $H$.  This implies non-trivial results on the potential graph profiles for generalized planar Tur\'an numbers, which is a problem initiated by Blekherman and Shi~\cite{Planar}.

\subsection{Extremal Problems for Subsets of Integers}
Many problems in additive combinatorics can be stated in terms of finding large sets of integers $A$ which avoids non-trivial solutions to a system of linear equations $\c{L}$, with notable examples including $k$-AP free sets~\cite{szemeredi1975sets} and Sidon sets~\cite{erdos1941problem}.  Motivated by this and the generalized Tur\'an problem for graphs, we ask: what is the maximum number of solutions to a system of linear equations $\c{K}$ that a set $A$ with $|A|=n$ can have if it contains no non-trivial solution to another system $\c{L}$?  

There exist a few sporadic cases of this general problem in the literature.  Notably, Fox and Pohoata~\cite{fox2021sets} answered a question of Erd\H{o}s by determining how many $\ell$-AP's a set of $n$ integers without any $k$-AP can have, and an observation of Sanders~\cite{sanders2009three} implies that sets of a given size with large additive energy (or equivalently many solutions to the Sidon equation $x_1+x_2=x_3+x_4$) have many 3-AP's.  In this paper we give a more general framework which puts these problems on a common footing.

In what follows, we keep our definitions as abstract as possible both to move us closer in abstraction to our most general supersaturation results, as well as to avoid defining things like what a non-trivial solution to a system is.  We maintain  similar notation to the generalized Tur\'an setting for ease of comparison, though we add a subscript $\N$ in order to avoid potential confusion.

\begin{defn}
	For a set $F\sub 2^{\N}$ of sets of integers and a set $A\sub \N$, we define $\num_\N(F,A)$ to be the number of $e\in F$ with $e\sub A$.  For two sets $H,F\sub 2^{\N}$ and an integer $n$, we define $\ex_\N(n,H,F)$ to be the maximum value of $\num_\N(H,A)$ among all sets $A\sub \N$ with $|A|=n$ and with $\num_\N(F,A)=0$.  For technical convenience we let $\ex_\N(n,H,F)=0$ if no such $A$ exists.
\end{defn}
For example, if $F$ is the set of all $k$-AP's and if $H$ is the set of all $\ell$-AP's, then $\ex_\N(n,H,F)$ exactly asks for the maximum number of $\ell$-AP's that a set of size $n$ can have if it has no $k$-AP, recovering the problem of \cite{fox2021sets}.   Applying our general supersaturation machinery to this setup gives an analog of \Cref{generalized Turan supersaturation}.

\begin{thm}\label{systems}
	Let $h,f,f'$ be positive integers with $f\le f'$.  Let $H,F\sub 2^{\N}$ be such that every set in $H$ has at most $h$ elements, every $e\in F$ has $f\le |e|\le f'$, and $\ex_\N(n,H,F)$ is non-decreasing in $n$.  If $\be<h$ is such that $\ex_\N(n,H,F)=O(n^\be)$, then there exists some $k_0$ such that the following holds for any $A\sub \N$ with $|A|=n$ and $\num_\N(H,A)\ge k n^\be$ for $k\ge k_0$:
	\begin{itemize}
		\item[(i)] We have
		\[\num_\N(F,A)=\Omega\left( k^{\frac{f+h-1-\be}{h-\be}}n^{\be-h+1}\right).\]
		\item[(ii)] If $f\ge h-\be$, then 
		\[\num_\N(F,A)=\Omega\left(k^{\frac{f}{h-\be}}\right).\]
		\item[(iii)] If $f<h-\be$ and if there exists no $e\in F$ and $e'\in H$ with $e\sub e'$, then
		\[\num_\N(F,A)=\Omega\left(k^{\frac{2}{h-\be-\min\{f',\lceil h-\be\rceil-1\}+2}}\right).\]
	\end{itemize}
\end{thm}

Here we allow $H,F$ to have different sized elements since (non-trivial) solutions to a system of equations often allow for repeated elements. Something like the awkward hypothesis in (iii) is needed to obtain any non-trivial supersaturation result in general: if $F\sub 2^\N$ is such that there exists an infinite set $A\sub \N$ which contains exactly one element of $F$ as a subset and if $H=\{e\cup \{n\}:e\in F,\ n\in \N\}$, then $\num_\N(H,A)$ is infinite despite only having $\num_\N(F,A)=1$.

\subsection{Extremal Problems as (Weighted) Independent Sets}
Many problems in extremal combinatorics can be interpreted as determining the size of a largest independent set in an auxiliary hypergraph.  For example, the Tur\'an problem of asking for the maximum number of edges in an $n$-vertex $F$-free graph is the same as asking for the largest size of an independent set in the hypergraph $\mathcal{F}$ whose vertex set is $E(K_n)$ and whose hyperedges correspond to copies of $F$ in $K_n$.  This independent set perspective comes with many advantages, such as allowing one to make use of the powerful method of hypergraph containers.

Part of the goal of this paper is to establish a generalization of this independent set approach by considering independent sets which are weighted by another hypergraph.

\begin{defn}
	For a hypergraph $\c{F}$ (possibly with repeated edges) and a set $S\sub V(\c{F})$, we write $\c{F}[S]$ to denote the induced subhypergraph of $\c{F}$ which has vertex set $S$ and all edges of $\c{F}$ which are contained in $S$.  We say $I\sub V(\c{F})$ is an independent set of $\c{F}$ if $e(\c{F}[I])=0$ and we write $\c{I}(\c{F})$ to denote the set of independent sets of $\c{F}$.  If $\c{H},\c{F}$ are hypergraphs
	with $V(\c{H})=V(\c{F})$ then we say that $(\mathcal{H},\mathcal{F})$ is a \textit{hypergraph pair} and we define
	\[\al_{\c{H}}(\c{F})=\max_{I\in \c{I}(\c{F})} e(\mathcal{H}[I]).\]
	That is, $\al_{\c{H}}(\c{F})$ is the maximum number of edges an independent set of $\c{F}$ can induce in $\mathcal{H}$.
\end{defn}
For example, if $\c{H}$ is the complete 1-uniform hypergraph on $V(\c{F})$ then $\al_{\c{H}}(\c{F})$ is just the maximum cardinality of an independent set of $\c{F}$. 

Our motivation for $\al_{\c{H}}(\c{F})$ comes from Tur\'an problems where the usual independent set perspective is ineffective.  For example, the standard way to translate the generalized Tur\'an problem into an independent set problem would be to consider the hypergraph $\c{F}'$ whose vertex set consists of copies of $H$ in $K_n$ and whose hyperedges are sets of copies of $H$ whose union contains a copy of $F$ in $K_n$.   The structure of $\c{F}'$ is rather complex, and for most choices of $H,F$ the hypergraph will not be uniform.   A cleaner way to view the generalized Tur\'an problem is to consider a pair of hypergraphs $\c{H},\c{F}$ with vertex set $E(K_n)$ where the hyperedges of $\c{H},\c{F}$ encode copies of $H,F$ respectively.  The generalized Tur\'an problem can then be expressed as finding $\al_{\c{H}}(\c{F})$, which has the advantage of both hypergraphs $\c{H},\c{F}$ being uniform and much simpler in structure than $\c{F}'$.

Our other guiding example comes from the problem of counting $n$-vertex $F$-free graphs.  After a long history of sporadic results \cite{balogh2011number,balogh2011number2,morris2016number}, this problem was essentially resolved in full by Ferber, McKinley, and Samotij~\cite{ferber2020supersaturated}.  A key ingredient in their work is a supersaturation result which implies that graphs $G$ with many edges contain many copies of $F$, and crucially they did not prove this by working with the most natural hypergraph $\c{F}'$ whose vertex set is $E(G)$ and whose hyperedges are copies of $F$, but rather with respect to the hypergraph $\c{F}$ whose vertex set is $V(G)$ and whose hyperedges are copies of $F$.  From this perspective, the ``correct'' measure of independent sets of $\c{F}$ is not how large they are, but rather how many edges of $G$ they induce, i.e.\ how many edges of the hypergraph $\c{H}$ with vertex set $V(G)$ and edge set $E(G)$ these independent sets of $\c{F}$ induce.  As such, their proof approach can be naturally stated in the language of $\al_{\c{H}}(\c{F})$.

Motivated by this work of Ferber, McKinley, and Samotij, we focus our exploration of $\al_{\c{H}}(\c{F})$ by studying supersaturation results; that is, results of the form that if $S\sub V(\c{H})=V(\c{F})$ has $e(\c{H}[S])\gg \al_{\c{H}}(\c{F})$, then $S$ induces many edges in $\c{F}$.  It is too much to expect  meaningful results for arbitrary pairs $(\c{H},\c{F})$, and as such we restrict our attention to families of pairs with somewhat reasonable behaviors.
	
\begin{defn}
	For positive integers $h,f,f'$ with $f\le f'$, we say that a hypergraph pair $(\c{H},\c{F})$ is \textit{$(h,f,f')$-bounded} if every edge of $\c{H}$ has at most $h$ elements and every edge $e\in E(\c{F})$ has  $f\le |e|\le f'$.  We say a set $\c{P}$ of hypergraph pairs is \textit{hereditary} if for every $(\c{H},\c{F})\in \c{P}$ and $S\sub V(\c{H})=V(\c{F})$ we have $(\c{H}[S],\c{F}[S])\in \c{P}$.  For each integer $n$ we define \[\c{P}_n=\{(\c{H},\c{F})\in \c{P}: |V(\c{H})|=|V(\c{F})|=n\},\] \[\ex(n,\c{P})=\max\{\al_{\c{H}}(\c{F}):(\c{H},\c{F})\in \c{P}_n\}\cup \{0\}.\]
\end{defn}

For hereditary $\c{P}$, the supersaturation condition ``every $(\c{H},\c{F})\in \c{P}$ and $S\sub V(\c{H})=V(\c{F})$ with $e(\c{H}[S])$ large has $e(\c{F}[S])$ large'' is equivalent to the simpler condition ``every $(\c{H},\c{F})\in \c{P}$ with $e(\c{H})$ large has $e(\c{F})$ large'' since the induced subhypergraph pair $(\c{H}[S],\c{F}[S])$ is also in $\c{P}$.  Because of this, we state our results in terms of this simpler edge condition.  We prove two main supersaturation results in this setting, the first of which generalizes the approach of \cite{ferber2020supersaturated}.

\begin{thm}\label{FMS bound}
	Let $\c{P}$ be a hereditary family of $(h,f,f')$-bounded hypergraph pairs and $\be<h$ a real number such that $\ex(n,\c{P})=O(n^\be)$.  If there exists real numbers $c,c'>0$ and $d<h$ such that $\ex(n,\c{P})\ge ce(\c{H})-c'e(\c{F})n^d$ for every $(\c{H},\c{F})\in \c{P}_n$ and every $n$, then there exists some $k_0$ such that any $(\c{H},\c{F})\in \c{P}_n$ with $e(\c{H})\ge k n^\be$ and $k\ge k_0$ has
	\[e(\c{F})=\Omega\left(k^{\frac{f+d-\be}{h-\be}}n^{\be-d}\right).\]
\end{thm}

The technical parameters $c,c',d$ roughly corresponds to the existence of an effective deletion argument for the problem at hand.  We will show that the conditions of this proposition are always satisfied when $d=h-1$, giving a general bound.
\begin{cor}\label{FMS corollary}
	Let $\c{P}$ be a hereditary family of $(h,f,f')$-bounded hypergraph pair and $\be<h$ a real number such that $\ex(n,\c{P})=O(n^\be)$.  There exists some $k_0$ such that any $(\c{H},\c{F})\in \c{P}_n$ with $e(\c{H})\ge k n^\be$ and $k\ge k_0$ has
	\[e(\c{F})=\Omega\left( k^{\frac{f+h-1-\be}{h-\be}}n^{\be-h+1}\right).\]
\end{cor}

These bounds are best possible if $\c{H}$ has $e(\c{H})=\Theta(n^{h})$ and if $\c{F}$ is $f$-uniform.  For small $k$ these bounds can be weak, and in particular if $d>\be$ then $k$ needs to be at least polynomial in $n$ in order for \Cref{FMS bound} to guarantee at least 2 edges of $\c{F}$ despite $e(\c{H})=kn^\be$ being polynomial larger than the extremal function $\ex(n,\c{P})$ which forces at least 1 edge.  Our second main result which uses a more involved argument largely gets around this type of issue.

\begin{thm}\label{supersaturation for f large}
	Let $\c{P}$ be a hereditary family of $(h,f,f')$-bounded hypergraph pairs, and let $\be<h$ be a real number such that $\ex(n,\c{P})=O(n^\be)$.  There exists some $k_0$ such that the following holds: if $f\ge h-\be$, then any $(\c{H},\c{F})\in \c{P}_n$ with $e(\c{H})\ge k n^\be$ and $k\ge k_0$ has
	\[ e(\c{F})=\Omega\left(k^{\frac{f}{h-\be}}\right).\]
	If $f<h-\be$ and if there does not exist $(\c{H},\c{F})\in \c{P}$ and $e\in \c{F},e'\in \c{H}$ with $e\sub e'$, then any $(\c{H},\c{F})\in \c{P}_n$ with $e(\c{H})\ge k n^\be$ and $k\ge k_0$ has
	\[e(\c{F})=\Omega\left(k^{\frac{2}{h-\be-\min\{f',\lceil h-\be\rceil-1\}+2}}\right).\]
\end{thm}

The first bound of \Cref{supersaturation for f large} has a number of appealing features: it is best possible in the extremal case $e(\c{H})=\Theta(n^h)$ when $\c{F}$ is $f$-uniform, it improves upon \Cref{FMS bound} precisely in the range $d>\be$ where \Cref{FMS bound} is ineffective for small $k$, and there exist examples where the bound is best possible in the case $\be=0$ for all values of $k$, such as the examples discussed around \Cref{generalized Turan supersaturation}(ii).  These two bounds combine to give the intuitive (but not so easy to prove) fact that if $e(\c{H})\gg \ex(n,\c{P})$ then $e(\c{F})\gg 1$ provided no edge of $\c{F}$ is contained in an edge of $\c{H}$, and such a condition is necessary in general to have non-trivial supersaturation results.

\textbf{Organization}. We show how to use these general supersaturation results to derive our applications in \Cref{sec:generalized Turan}.  We then give an easy proof of \Cref{FMS bound} and \Cref{FMS corollary} in \Cref{sec:FMS}.  We prove Theorem~\ref{supersaturation for f large} in \Cref{sec:main proofs}, which is the main technical part of the paper.  We close by discussing some further extensions of our results in \Cref{sec:concluding}.

\textbf{AI Declaration}.  We used Claude's Sonnet 4.8 to conduct literature searches and correct for typos.  All of the mathematics and writing is our own.

\section{Proof of Applications}\label{sec:generalized Turan}
We begin by proving our supersaturation results for subsets of integers assuming our general supersaturation theorems.

\begin{proof}[Proof of \Cref{systems}]
	Recall that we are considering $H,F\sub 2^{\N}$ such that every set in $H$ has at most $h$ elements, every $e\in F$ has $f\le |e|\le f'$, $\ex_\N(n,H,F)$ is non-decreasing in $n$, and $\be<h$ is such that $\ex_\N(n,H,F)=O(n^\be)$. For every set of integers $A$, let $\c{H}_A$ be the hypergraph with $V(\c{H}_A)=A$ where a set $e\sub A$ is an edge of $\c{H}_A$ if and only if $e\in H$.  We similarly define $\c{F}_A$ and let $\c{P}=\{(\c{H}_A,\c{F}_A):A\sub \N\}$.
	
	Our assumptions on $H,F$ imply that each pair in $\c{P}$ is $(h,f,f')$-bounded.  Moreover, we observe that for any set $A\sub \N$ and $S\sub A$ we have $\c{H}_A[S]=\c{H}_{A\cap S}$ and $\c{F}_A[S]=\c{F}_{A\cap S}$, so $\c{P}$ is a hereditary family.  The last ingredient we need is the following
	
	\begin{claim}
		For all $n$ we have $\ex(n,\c{P})\le\ex_\N(n,H,F)$.
	\end{claim}
	In fact it is easy to show equality holds, but we will not need this.  For this proof we use that $\num_\N(H,A)=e(\c{H}_A)$ and $\num_\N(F,A)=e(\c{F}_A)$ by definition of these hypergraphs and $\num_\N$.  
	\begin{proof}
		This is true if $\ex(n,\c{P})=0$, and otherwise by definition there exists some $(\c{H}_A,\c{F}_A)\in \c{P}_n$ such that $\ex(n,\c{P})=\al_{\c{H}_A}(\c{F}_A)$, noting that this implies $|A|=|V(\c{H})|=n$.  By definition this means there exists some independent set $S\sub A=V(\c{F}_A)$ of $\c{F}_A$ with \[\ex(n,\c{P})=e(\c{H}_A[S])=e(\c{H}_{A\cap S})=\num_\N(H,A\cap S)\le \ex_\N(|A\cap S|,H,F),\]  
		with this last step using that $\num_\N(F,A\cap S)=0$ since $S$ is an independent set of $\c{F}_A$.  Since $|A\cap S|\le |A|=n$ and we assumed $\ex_\N(n,H,F)$ is non-decreasing in $n$, this implies $ \ex(n,\c{P})\le \ex_\N(n,H,F)$.
	\end{proof}
	This claim and our hypothesis implies $\ex(n,\c{P})=O(n^\be)$.  We thus find that $\c{P}$ satisfies all the hypothesis of \Cref{FMS corollary}, and as such if $k_0$ is as in this corollary, then for any $A\sub \N$ with $|A|=n$ and $\#_\N(H,A)\ge k n^\be$ with $k\ge k_0$ we have
	\[\#_\N(F,A)=e(\c{F}_A)=\Om(k^{\frac{f+h-1-\be}{h-\be}}n^{\be-h+1}).\]
	This gives (i), and one similarly derives (ii) and (iii) by using \Cref{supersaturation for f large}.
\end{proof}

We next prove our results for generalized Tur\'an problems, which requires a bit more work.

\begin{proof}[Proof of \Cref{generalized Turan supersaturation}]
	Let $H$ be an $h$-vertex hypergraph, $F$ an $f$-vertex hypergraph without isolated vertices, and $\be<h$ such that $\ex(n,H,F)=O(n^\be)$.  For every graph $G$, define the $h$-uniform hypergraph $\c{H}_G$ to have vertex set $V(G)$ where a set $S$ of $h$ vertices in $V(G)$ forms a hyperedge of $\c{H}_G$ if and only if $G[S]$ contains a copy of $H$.  Similarly define $\c{F}_G$ to be the $f$-uniform hypergraph on $V(G)$ encoding copies of $F$ and let $\c{P}=\{(\c{H}_G,\c{F}_G):G\textrm{ is a graph}\}$.  
	
	Observe that for a graph $G$ and a set $S\sub V(G)$, we have $\c{H}_G[S]=\c{H}_{G[S]}$ and $\c{F}_G[S]=\c{F}_{G[S]}$.  It follows that $\c{P}$ is a hereditary family of $(h,f,f)$-bounded hypergraph pairs.  Moreover, we have the following.
	\begin{claim}
		For all $n$,
		\[\frac{1}{h!} \ex(n,H,F)\le \ex(n,\c{P})\le \ex(n,H,F).\]
	\end{claim}
	For this and what follows, we make frequent use of the fact that
	\[e(\c{H}_G) \le \num(H,G)\le h! e(\c{H}_G),\]
	since each edge $S$ of $\c{H}_G$ corresponds to at least 1 copy of $H$ in $G[S]$ and trivially at most $h!$ copies in $G[S]$; we similarly use that $e(\c{F}_G)\ge \num(F,G)$.
	\begin{proof}
		For the lower bound on $\ex(n,\c{P})$, let $G$ be an $n$-vertex $F$-free graph with $\ex(n,H,F)$ copies of $H$.  Observe that $e(\c{F}_G)=0$ since $G$ is $F$-free, and hence $V(G)$ is an independent set of $\c{F}_G$, implying that $\al_{\c{H}_G}(\c{F}_G)=e(\c{H}_G)$. It follows that 
		\[\al_{\c{H}_G}(\c{F}_G)=e(\c{H}_G)\ge \frac{1}{h!} \num(H,G)=\frac{1}{h!} \ex(n,H,F),\]
		and since $(\c{H}_G,\c{F}_G)\in \c{P}_n$ because $|V(G)|=n$ this implies the desired lower bound.
		
		For the upper bound, let $G$ be an $n$-vertex graph such that $\alpha_{\c{H}_G}(\c{F}_G)=\ex(n,\c{P})$, which means there exists some $S\sub V(G)$ such that $\c{F}_G[S]$ contains no edges of $\c{F}$ and $e(\c{H}_G[S])=\ex(n,\c{P})$.  That $\c{F}_G[S]$ has no edges is equivalent to saying $G[S]$ is $F$-free, and as such \[\ex(|S|,H,F)\ge \num(H,G[S])\ge e(\c{H}_{G[S]})=e(\c{H}_G[S])=\ex(n,\c{P}).\]
		This implies $\ex(n,\c{P})\le \max_{n'\le n} \ex(n',H,F)$, and because $F$ has no isolated vertices we have $\ex(n',H,F)\le \ex(n,H,F)$ for all $n'\le n$ (simply by taking an extremal $n'$-vertex graph and adding $n-n'$ isolated vertices to it), from which the bound follows.
	\end{proof}
	With this and our hypothesis $\ex(n,H,F)=O(n^\be)$, we see that under the hypothesis of bound (ii) of this theorem that the first part of \Cref{supersaturation for f large} holds for $\c{P}$ with respect to some $k_0$.  As such, if $\num(H,G)\ge k n^\be$ for some $k\ge h!\cdot k_0$, we have $e(\c{H}_G)\ge (h!)^{-1} k n^\be$ and hence by \Cref{supersaturation for f large},
	\[\num(F,G)\ge e(\c{F}_G)=\Omega(k^{\frac{f}{h-\be}}),\]
	giving (ii).
	
	For (i) we wish to apply \Cref{FMS bound}.  We observe for any $n$-vertex graph $G$ that
	\[\ex(n,H,F)\ge \num(H,G)-n^{h-2} \num(F,G)\ge e(\c{H}_G)-f! n^{h-2}e(\c{F}_G),\]
	with the first inequality following simply because we can iteratively delete edges of $G$ which are contained in a copy of $F$ with each deleted edge destroying at most $n^{h-2}$ copies of $H$ in $G$.  From this and the lower bound of the claim it follows that the hypothesis of \Cref{FMS bound} holds for $\c{P}$ with $c=(h!)^{-1}$, $c'=f! (h!)^{-1}$ and $d=h-2$, and the same reasoning as in (ii) gives (i).
	
	For (iii), assume $f<h-\be$ and that $F$ is not isomorphic to a subgraph of $H$.  We ultimately wish to apply the second part of \Cref{supersaturation for f large}, and for this we need to consider a slightly more complex family than $\c{P}$.  To this end, let $V(H)=\{x_1,\ldots,x_h\}$.  Given a graph $G$, we say that a partition $\c{V}=\{V_1,\ldots,V_h\}$ is an $H$-partition of $G$ if having an edge $uv\in E(G)$ with say $u\in V_i$ and $v\in V_j$ implies $x_ix_j\in E(H)$ (in other words, $G$ has an $H$-partition if and only if $G$ is a subgraph of a blowup of $H$).  Given a graph $G$ and $H$-partition $\c{V}$ we let $\c{H}_{G,\c{V}}$ be the hypergraph on $V(G)$ where a set of vertices $\{u_1,\ldots,u_h\}\sub V(G)$ forms a hyperedge if and only if $u_i\in V_i$ for all $i$ and there exists a copy of $H$ in $G$ on $\{u_1,\ldots,u_h\}$ with $u_i$ playing the role of $x_i$ for all $i$. Let $\c{P}'$ be the set of all pairs $(\c{H}_{G,\c{V}},\c{F}_G)$ with $G$ a graph and $\c{V}$ an $H$-partition of $G$.  
	
	 Observe that $\c{P}'$ is hereditary since $\c{H}_{G,\c{V}}[S]=\c{H}_{G[S],\c{V}'}$ with $\c{V}'=\{V_1\cap S,\ldots,V_h\cap S\}$ and also that each hypergraph pair in $\c{P}'$ is $(h,f,f)$-bounded.  Also observe that $\c{H}_{G,\c{V}}\sub \c{H}_G$ for all $G,\c{V}$, so it follows that $\al_{\c{H}_{G,\c{V}}}(\c{F}_G)\le \al_{\c{H}_G}(\c{F}_G)$ and hence $\ex(n,\c{P}')\le \ex(n,\c{P})=O(n^\be)$.
	
	We claim that if $(\c{H}_{G,\c{V}},\c{F}_G)\in \c{P}'$ then no edge  $e\in E(\c{F}_G)$ is contained in any edge $e'\in E(\c{H}_{G,\c{V}})$.  Assume for contradiction that this was false for some $e,e'$.  Note that any $e'\in E(\c{H}_{G,\c{V}})$ intersects each $V_i\in \c{V}$ in exactly 1 vertex, so $e\sub e'$ implies that we must have, say, $e=\{u_1,\ldots,u_f\}$ with $u_i\in V_i$ for all $i$.  By definition of $e\in E(\c{F}_G)$, we have that $G[e]$ contains a copy of $F$.  Because of this, we can write $V(F)=\{y_1,\ldots,y_f\}$ such that for every $y_iy_j\in E(F)$ we have $u_i u_j\in E(G)$, and by definition of $H$-partitions this implies $x_ix_j\in E(H)$.  This in turn then means that $H$ contains a subgraph isomorphic to $F$ on $\{x_1,\ldots,x_f\}$, a contradiction to our assumption that $F$ is not isomorphic to any subgraph of $H$.
	
	With the above we see that $\c{P}'$ satisfies the conditions of the second half of \Cref{supersaturation for f large}.  Now consider any $n$-vertex graph $G$ with $\num(H,G)\ge k n^\be$ for $k$ large.  For a partition $\c{V}=\{V_1,\ldots,V_h\}$, let $G_{\c{V}}$ be the subgraph of $G$ obtained by keeping only the edges of $G$ which lie between parts $V_i,V_j$ with $x_ix_j\in E(H)$, which in particular means $G_{\c{V}}$ has $H$-partition $\c{V}$.  One can show that a random choice of $\c{V}$ satisfies in expectation that $e(\c{H}_{G_{\c{V}},\c{V}})\ge h^{-h}  \num(H,G)\ge h^{-h} k n^\be$, and by choosing a $\c{V}$ which achieves this bound and using \Cref{supersaturation for f large} we find that
	\[\num(F,G)\ge \num(F,G_{\c{V}})\ge e(\c{F}_{G_\c{V}})=\Omega(k^{\frac{2}{h-\be-f+2}}),\]
	proving the result.
\end{proof}
Nearly identical arguments give our result for generalized planar Tur\'an numbers.

\begin{proof}[Sketch of proof of \Cref{generalized planar Turan}]
	We define $\c{H}_G,\c{F}_G$ as in the proof of \Cref{generalized Turan supersaturation}, except now we define $\c{P}$ to consist of all pairs $(\c{H}_G,\c{F}_G)$ where $G$ is a planar graph.  We have that $\c{P}$ is a hereditary family because $\c{H}_{G}[S]=\c{H}_{G[S]}$ and crucially because $G[S]$ is planar whenever $G$ is since planar graphs are hereditary.  From here the exact same details as before go through where for the $h$-partite result we implicitly use that $G_{\c{V}}$ is planar since planar are closed under removing edges.
\end{proof}

Even more generally, for any family of graphs $\c{G}$ one can define a generalized Tur\'an function $\ex_{\c{G}}(n,H,F)$ to be the maximum number of copies of $H$ in an $n$-vertex $F$-free graph of $\c{G}$, and if $\c{G}$ is closed under deleting vertices and edges then these exact same proofs give the same supersaturation results as in \Cref{generalized Turan supersaturation} for $\ex_{\c{G}}(n,H,F)$ in place of $\ex(n,H,F)$.

\section{Proof of \Cref{FMS bound}}\label{sec:FMS}
We begin with our simplest proof which is based on the argument of Ferber, McKinley, and Samotij.  The idea is to look at a $p$-random induced subhypergraph of $(\c{H},\c{F})$ and then perform a deletion argument whose effectiveness is dictated by the parameter $d$.
\begin{proof}[Proof of \Cref{FMS bound}]
	Recall that we wish to show if $\c{P}$ is a hereditary family of $(h,f,f')$-bounded hypergraph pairs with $\ex(n,\c{P})=O(n^\be)$ for some $\be<h$, and if there exist real numbers $c,c'>0$ and $d<h$ such that every $(\c{H},\c{F})\in \c{P}_n$ has $\al_{\c{H}}(\c{F})\ge ce(\c{H})-c'e(\c{F})n^d$, then there exists some $k_0$ such that every $(\c{H},\c{F})\in \c{P}_n$ with $e(\c{H})\ge k n^\be$ for $k\ge k_0$ satisfies
	\[e(\c{F})=\Omega\left( k^{\frac{f+d-\be}{h-\be}}n^{\be-d}\right).\]
	
	Let $k_0,C$ be large constants to be specified later, and assume for contradiction that there exists $(\c{H},\c{F})\in \c{P}_n$ with $e(\c{H})\ge k n^\be$ for $k\ge k_0$ such that \begin{equation}e(\c{F})<C^{-f}\cdot k^{\frac{f+d-\be}{h-\be}}n^{\be-d}.\label{eq:FMS}\end{equation}  Choose a real number $p$ with $Ck^{-\frac{1}{h-\be}}\le p\le 2Ck^{-\frac{1}{h-\be}}$ such that $pn$ is an integer, and we will chose our constants $k_0,C$ such that $p\le 1$ and $pn\ge h$.  Let $V_p\sub V(\c{H})=V(\c{F})$ be a set of $pn$ vertices chosen uniformly at random amongst all subsets of size $pn$.  Consider the induced subhypergraphs $\c{H}_p=\c{H}[V_p]$ and $\c{F}_p=\c{F}[V_p]$.  By using linearity of expectation, $pn\ge h$, and that every edge of $\c{H}$ has size at most $h$ we find
	\[\E[e(\c{H}_p)]=\Omega(p^{h} e(\c{H}))=\Omega(C^{h} k^{\frac{-\be}{h-\be}} n^\be), \]
	and similarly 
	 \eqref{eq:FMS} and that every edge of $\c{F}$ has size at least $f$ implies
	\[\E[e(\c{F}_p)]=O(p^{f} e(\c{F}))=O( k^{\frac{d-\be}{h-\be}}n^{\be-d}).\]
	Because $\c{P}$ is hereditary, we have $(\c{H}_p,\c{F}_p)\in \c{P}_{pn}$, and hence by hypothesis
	\[\ex(pn,\c{P})\ge ce(\c{H}_p)-c'e(\c{F}_p)(pn)^d= ce(\c{H}_p)-c'e(\c{F}_p) (2C)^{d}k^{\frac{-d}{h-\be}}n^{d}.\]
	By taking the expectation of this lower bound we find for $C$ sufficiently large and $d<h$ that
	\[\ex(pn,\c{P})=\Om(C^{h}k^{\frac{\be}{h-\be}} n^{\be}).\]
	On the other hand, we have by assumption that $\ex(pn,\c{P})=O(p^\be n^\be)=O(C^{\be}k^{\frac{\be}{h-\be}} n^{\be})$, and this gives a contradiction for $C$ sufficiently large since $\be<h$, proving the result.
\end{proof} 

We now formally show that the corollary of \Cref{FMS bound} holds.

\begin{proof}[Proof of \Cref{FMS corollary}]
	The result will follow immediately from \Cref{FMS bound} if we can show that every family $\c{P}$ of $(h,f,f')$-bounded hypergraph pairs satisfies $\ex(n,\c{P})\ge e(\c{H})-e(\c{F})n^{h-1}$ for every $(\c{H},\c{F})\in \c{P}_n$.  And indeed, take $I=V(\c{H})=V(\c{F})$ and iteratively remove vertices from $I$ until $e(\c{F}[I])=0$.  We have that $I$ is an independent set of $\c{F}$ by definition, and hence
	\[\ex(n,\c{P})\ge \al_{\c{H}}(\c{F})\ge e(\c{H}[I])\ge e(\c{H})-e(\c{F})n^{h-1},\]
	where this last inequality uses that we trivially removed at most $e(\c{F})$ vertices going from $V(\c{H})$ to $I$ and that every such removal deletes at most $n^{h-1}$ edges from $\c{H}$ since $\c{H}$ being $n$-vertex and having edge sizes at most $h$ implies that any given vertex is in at most $n^{h-1}$ edges of $\c{H}$.
\end{proof}

\section{Proof of \Cref{supersaturation for f large}}\label{sec:main proofs}

In this section we restrict to a partite setting similar to our proof of \Cref{generalized Turan supersaturation}(iii).

\begin{defn}
	For a hypergraph $\c{H}$, we say that sets of vertices $V_1,\ldots,V_{h}$ are an \textit{$h$-partition} of $\c{H}$ if every edge intersects each $V_i$ set in exactly 1 vertex, and we say $\c{H}$ is \textit{$h$-partite} if there exists an $h$-partition of $\c{H}$.
\end{defn}
Note that $h$-partite hypergraphs are automatically $h$-uniform.  We will reduce our problem to partite hypergraphs through the following.

\begin{lem}\label{partite}
	If $\c{H}$ is a hypergraph where every edge $e\in E(\c{H})$ has size at most $h$ and if $e(\c{H})\ge 2$, then there exists a subgraph $\hat{\c{H}}\sub \c{H}$ on the same vertex set of $\c{H}$ which is $\hat{h}$-partite for some $1\le \hat{h}\le h$ and which has $e(\hat{\c{H}})\ge \half h^{-h-1} e(\c{H})$. 
\end{lem}
\begin{proof}
	For each $0\le \hat{h}\le h$, let $\c{H}_{\hat{h}}\sub \c{H}$ denote the hypergraph on $V(\c{H})$ consisting of the edges of size $\hat{h}$ of $\c{H}$.  Note that $\sum_{\hat{h}=1}^h e(\c{H}_{\hat{h}})\ge e(\c{H})-1\ge \half e(\c{H})$ since $\c{H}_0$ has at most 1 edge and $\c{H}$ has at least 2 edges.  By the pigeonhole principle there is some $1\le \hat{h}\le h$ such that $e(\c{H}_{\hat{h}})\ge \half h^{-1} e(\c{H})$.  By a standard probabilistic argument there exists an $\hat{h}$-partite subgraph $\hat{\c{H}}\sub \c{H}_{\hat{h}}$ with $e(\hat{\c{H}})\ge (\hat{h})^{-\hat{h}} e(\c{H}_{\hat{h}})\ge \half h^{-h-1} e(\c{H})$, proving the result.
\end{proof}
This lemma will ultimately require us to work with $\hat{h}$-partite hypergraphs, but for ease of notation we will consider $h$-partite hypergraphs until needed.  Because our forthcoming proof is a little technical, we will spend a brief moment going over its motivating ideas.

The idea for our proof of \Cref{supersaturation for f large} is similar to  \Cref{FMS bound}: we again take $p$-random induced subhypergrphs $\c{H}_p,\c{F}_p$ and argue that if $e(\c{F})$ is small then $\al_{\c{H}_p}(\c{F}_p)\gg \ex(pn,\c{P})$, a contradiction.  However, we want to do this while avoiding any assumption of the form $\al_{\c{H}}(\c{F})\ge e(\c{H})-e(\c{F})n^d$ to delete the remaining edges of $\c{F}_p$, both because such a bound may not exist, and also because even when such an inequality holds the bounds they give can be rather weak.

To get around this, we will try to choose $p$ small enough so that $\c{F}_p$ has 0 edges with high probability.  To sketch the details further, suppose $\ex(n,\c{P})=O(n^\be)$ and assume for contradiction that there exist $(\c{H},\c{F})\in \c{P}_n$ with $e(\c{H})\ge k n^\be$ and $e(\c{F})\ll k^{\frac{f}{h-\be}}$.  Define $V_p\sub V(\c{H})=V(\c{F})$ by keeping each vertex independently and with probability $p\gg k^{-\frac{1}{h-\be}}$ and let $\c{H}_p,\c{F}_p$ be the subhypergraphs of $\c{H},\c{F}$ induced by $V_p$.  One can show using Markov and Chernoff that $e(\c{F}_p)=0$  and $|V_p|\approx pn$ with high probability.  If we \textit{pretend} also that $e(\c{H}_p)$ is close to $\E[e(\c{H}_p)]\approx p^{h}  kn^\be$, then $\ex(pn,\c{P})\ge \al_{\c{H}_p}(\c{F}_p)=e(\c{H}_p)=\Om(p^{h} kn^\be)$, contradicting $\ex(pn,\c{P})=O(p^\be n^\be)$ for $p\gg k^{\frac{-1}{h-\be}}$.

The place where this proof goes wrong is in our assumption $e(\c{H}_p)\approx \E[e(\c{H}_p)]$ which can fail with high probability.  Indeed, if $\c{H}$ is $h$-partite with $|V_i|\ll p^{-1}$ for some $i$, then with high probability $\c{H}_p$ will contain no vertices of $V_i$ and hence no edges.

The idea for getting around this particular issue is that if $\c{H}$ is $h$-partite then we will try to shrink each of our parts $V_i$ by different factors $q_i$ based on how  ``large'' each part is.  More precisely, we will attempt to take $q_i$ equal to some target value $p$ (such as $p\approx k^{\frac{-1}{h-\be}}$ in our proof sketch), and if it is not possible to shrink $V_i$ by this much then we will instead shrink $V_i$ down to a single vertex and record this lesser shrinking by setting $q_i\approx |V_i|^{-1}$.  Doing this will introduce a number of technical complications, one of which is that the probability of a given edge $e\in E(\c{F})$ surviving in our random subgraph will be a function of how $e$ intersects the various parts $V_i$.  To deal with this we need a definition.

\begin{defn}\label{type definition}
	Let $(\c{H},\c{F})$ be a hypergraph pair such that $\c{H}$ has an $h$-partition $V_1,\ldots,V_{h}$.  For each $e\in E(\c{F})$ we define the  \textit{type} of $e$ to be the tuple $\mathbf{t}=(|e\cap V_1|,\ldots,|e\cap V_{h}|)$.  For a given sequence of non-negative integers $\mathbf{t}=(t_1,\ldots,t_{h})$, we let $e_{\mathbf{t}}(\c{F})$ denote the number of edges of $\c{F}$ which have type $\mathbf{t}$.  If $\c{F}'=\c{F}[S]$ then we define the type of each edge $e\in E(\c{F}')$ with respect to the partition $V_1\cap S,\ldots,V_{h}\cap S$.
\end{defn}

We will make frequent use of the fact that if every edge of $\c{F}$ has size at most $f'$, then $e_{\mathbf{t}}(\c{F})>0$ only holds for $\mathbf{t}$ with $t_i\le f'$ for all $i$, and in particular there exist at most $h^{f'}$ such types.
We now state our technical result for how to try to shrink a part $V_i$ by some target quantity $p$ while maintaining many edges of $\c{H}$, few edges of $\c{F}$, and control over the shrunken size of $V_i$.

\begin{lem}\label{shrinkTechnical}
	Let $(\c{H},\c{F})$ be a hypergraph pair such that $\c{H}$ has an $h$-partition $V_1,\ldots,V_{h}$ and every edge of $\c{F}$ has size at most $f'\ge 1$. For any $i\in \{1,\ldots,h\}$ and real number $p\in [0,1]$, there exists a subset $V_i'\sub V_i$ and a real number $1\ge q\ge p$ such that the following holds for the induced subhypergraphs $\c{H}'=\c{H}[\bigcup_{j\ne i} V_j\cup V_i']$ and $\c{F}'=\c{F}[\bigcup_{j\ne i} V_j\cup V_i']$:
	\begin{itemize}
		\item[(a)] $e(\c{H}')\ge q e(\c{H})$.
		\item[(b)] For each type $\mathbf{t}$, we have \[e_\mathbf{t}(\c{F}')\le (40h)^{f'} q^{t_i} e_\mathbf{t}(\c{F}).\]
		\item[(c)] If $|V'_i|\ge 2$ then $q=p$ and $|V'_i|\le 20 p|V_i|$.
	\end{itemize}
\end{lem}
\begin{proof}
	If $p\ge 1/20$ then the result trivially holds for $V'_i=V_i$ and $q=p$, so we may assume from now on that $p\le 1/20$.  
	
	We split our argument into two cases, starting with the case that there exists a small set of vertices in $V_i$ which are in many edges of $\c{H}$.  To this end, for each $v\in V_i$, let $\deg_{\c{H}}(v)$ denote the number of edges in $\c{H}$ which use $v$ and let $q_v=\frac{\deg_{\c{H}}(v)}{e(\c{H})}$, noting that $\sum_v q_v=1$ since $V_1,\ldots,V_{h}$ is an $h$-partition (and hence each edge uses a unique vertex in $V_i$).   Let $L=\{v:q_v\ge p\}$.
	
	First consider the case $\sum_{v\in L} q_v\ge \half$.  In this situation we aim to take $V'_i$ to be a single vertex $v\in L$ and to set $q=q_v\ge p$.  By assumption, any such $V'_i=\{v\}$ will satisfy (a) and (c), so we will complete the proof once we show the following.
	\begin{claim}
		There exists some $v\in L$ such that $V'_i=\{v\}$ satisfies (b) with $q=q_v$.
	\end{claim}
	\begin{proof}
		Observe that for any choice of $V'_i=\{v\}$ we have $e_{\mathbf{t}}(\c{F}')=0$ whenever $t_i\ge 2$ since trivially each edge $e\in E(\c{F}')$ can intersect $V'_i$ in at most one vertex.  Similarly, we always have $e_{\mathbf{t}}(\c{F}')=e_{\mathbf{t}}(\c{F})$ if $t_i=0$, so it remains only to check those $\mathbf{t}$ with $t_i=1$.  For such $\mathbf{t}$, let $\deg_{\mathbf{t}}(v)$ denote the number of edges in $\c{F}$ of type $\mathbf{t}$ which $v$ is contained in and let $B_{\mathbf{t}}$ be the set of $v\in L$ with $\deg_{\mathbf{t}}(v)\ge (40h)^{f'} q_v e_{\mathbf{t}}(\c{F})$.  Because $t_i=1$, each edge counted by $e_{\mathbf{t}}(\c{F})$ contains a unique vertex in $V_i$, which means \[\sum_{v\in B_{\mathbf{t}}} (40h)^{f'} q_v e_\mathbf{t}(\c{F}) \le \sum_{v\in B_{\mathbf{t}}} \deg_{\mathbf{t}}(v) \le e_{\mathbf{t}}(\c{F}),\] or equivalently
		\[\sum_{v\in B_{\mathbf{t}}} q_v\le (40 h)^{-f'}.\]
		By summing this bound over the at most $h^{f'}$ choices of types $\mathbf{t}$ for which $e_{\mathbf{t}}(\c{F})>0$, we find that the sum of $q_v$ over all $v\in \bigcup_{\mathbf{t}} B_{\mathbf{t}}$ is at most $(40)^{-f'}<\half$.  On the other hand, we assumed for this case that $\sum_{v\in L} q_v\ge \half$, so there must exists some $v\in L\sm \bigcup_{\mathbf{t}} B_{\mathbf{t}}$.  We then have that $V'_i=\{v\}$ satisfies (b) with $q=q_v$ by definition of the $B_{\mathbf{t}}$, proving the claim.
	\end{proof}
	This establishes the first case of our argument, so we may assume from now on that $\sum_{v\notin L} q_v\ge \half$.
	
	Let $V'_i\sub V_i$ be the random set obtained by keeping each element of $V_i\sm L$ independently and with probability $10p$.  We aim to show this set with $q=p$ satisfies the conditions of the lemma with positive probability. 
	
	For (c), we make use of a standard form of the Chernoff bound which says that for a binomial random variable $X$ with mean $\mu$ and for any $0<\del<1$ we have \[\Pr[X\ge (1+\del)\mu]\le e^{-\frac{\del^2 \mu}{2+\del}}.\] Observe that $|V'_i|$ is a binomial random variable with $\E[|V'_i|]=10p |V_i\sm L|\le 10 p |V_i|$. Moreover, we must have $|V_i\sm L|\ge \half p^{-1}$, since otherwise by definition of $L$ we would have $\sum_{v\notin L} q_v<p\cdot \half p^{-1}$, a contradiction.  As such,
	\[\Pr\left[|V_i'|\le 20 p |V_i|\right]\ge \Pr\left[|V_i'|\le 2 \E[|V_i'|]\right]\ge 1- e^{-\frac{10 p \cdot \half p^{-1}}{3}}>\frac{2}{3}.\]

	With (b) in mind, for each type $\mathbf{t}$ we observe that $\E[e_{\mathbf{t}}(\c{F}')]=(10p)^{t_i} e_{\mathbf{t}}(\c{F})$ since a given edge of type $\mathbf{t}$ survives in $\c{F}'$ only if each of its $t_i$ vertices survive in $V'_i$.  Using that $t_i\le f'$ for all $\mathbf{t}$ with $e_{\mathbf{t}}(\c{F})>0$  and Markov's inequality gives
	\begin{align*}\Pr[e_{\mathbf{t}}(\c{F}')> (40 h)^{f'} p^{t_i} e_{\mathbf{t}}(\c{F})] &\le \Pr[e_{\mathbf{t}}(\c{F}')> (4h)^{f'} \E[e_{\mathbf{t}}(\c{F}')]\\ &\le  (4h)^{-f'}.\end{align*}
	Since the event that (b) fails is the union of at most $h^{f'}$ such events, we conclude that (b) holds with probability at least $1-4^{-f'}>\frac{2}{3}$.

	It remains to handle (a), for which we use Chebyshev's inequality applied to $e(\c{H}')$.  Observe that \begin{equation}\E[e(\c{H}')]=\sum_{v\notin L} 10 p\cdot  \deg_{\c{H}}(v)=10 p e(\c{H})\sum_{v\notin L} q_v\ge 5 p e(\c{H}),\label{eq:mean}\end{equation}
	with this last step using the hypothesis of the case we are in.  On the other hand, since $e(\c{H}')$ is the sum of independent Bernoulli random variables which take on value $\deg(v)$ with probability $10p$ and since $p\le 1/20$, we have
	\begin{equation}\Var(e(\c{H}'))=\sum_{v\notin L} \deg_{\c{H}}(v)^2\cdot (10p -(10p)^2)\le  \sum_{v\notin L} \deg_{\c{H}}(v)^2\cdot 5 p=5 p e(\c{H})^2\sum_{v\notin L} q_v^2.\label{eq:variance}\end{equation}
	To bound this, we observe an analytic fact.
	\begin{claim}
		If $x_1,x_2\ldots$ and $p$ are real numbers with $0\le x_j\le p\le 1$ for all $j$ and with $\sum_j x_j\le 1$, then $\sum_j x_j^2\le p+p^2$.
	\end{claim}
	\begin{proof}
		If there exist $x_j,x_{j'}$ and real $\ep>0$ with $\ep\le x_j\le x_{j'}\le p-\ep$, then we can replace $x_j,x_{j'}$ with $x_j-\ep,x_{j'}+\ep$ which maintains the hypothesis of the claim while increasing the square sum by $2\ep^2+2\ep (x_{j'}-x_j)>0$.  As such, $\sum x_j^2$ is maximized when as many terms are equal to $p$ as possible and when at most one term lies strictly between $0$ and $p$.  Since $\sum x_j\le 1$, there can be at most $ p^{-1} $ terms with $x_j=p$, and the remaining term has $x_j\le p$.  Thus in total the maximum is at most $p^2\cdot (p^{-1}+1)=p+p^2$ as desired.
	\end{proof}
	Note that $p+p^2\le 1.05 p$ since $p\le 1/20$, and this together with the claim and \eqref{eq:variance} implies
	\[\Var(e(\c{H}'))\le 5.25 p^2 e(\c{H})^2.\]  This inequality, \eqref{eq:mean}, and Chebyshev's inequality combine to give
	\[\Pr[|e(\c{H}')-\E[e(\c{H}')]|\ge .8 \E[e(\c{H}')]]\le \frac{ \Var(e(\c{H}'))}{.64 \E[e(\c{H}')]^2}\le \frac{1.05}{.64\cdot 5}<\frac{1}{3}.\]
	In particular, this and our bound on $\E[e(\c{H}')]$ from \eqref{eq:mean} implies that the probability $e(\c{H}')$ is at least $.2 \E[e(\c{H}')]\ge p e(\c{H})$ is more than $2/3$.
	
	Putting all this together, we see that our random choice of $V'_i$ satisfies each of (a), (b), and (c) individually with probability strictly greater than $2/3$, so with positive probability it satisfies all three conditions, giving the result.
\end{proof}

To implement our proof strategy,  if we are given some $(\c{H},\c{F})$ with $\c{H}$ having $h$-partition $V_1,\ldots,V_{h}$, then we would ideally like to iteratively apply \Cref{shrinkTechnical} to each $V_i$ with the same value $p$, and if \Cref{shrinkTechnical} applies with $q=p$ for all $i$ then this gives our result.  The problem is that if we get, say, $q=1$ for a number of $i$, then our upper bounds for $e_{\mathbf{t}}(\c{F}')$ from \Cref{shrinkTechnical} will be too weak.  

To get around this, we will start by trying to apply \Cref{shrinkTechnical} to shrink each $V_i$ by some initial factor $p$, but if we ever get some $q>p$ from \Cref{shrinkTechnical}, then we will change our goal to shrink each set by some new factor $p'$ with $p'$ a function of $q$.  We will then keep trying to shrink our sets by a factor of $p'$ unless again we get some $q'>p'$ from \Cref{shrinkTechnical}, in which case we change our target value to some $p''$ and so on.  Implementing this procedure with the optimal choices of $p,p',p'',\ldots$ ultimately gives the following technical statement.

\begin{prop}\label{shrinkingApplied}
	Let $(\c{H},\c{F})$ be a hypergraph pair of $n$-vertex hypergraphs such that $\c{H}$ has an $h$-partition $V_1,\ldots,V_{h}$ and every edge of $\c{F}$ has size at most $f'$, and let $k,C,\be$ be non-negative reals with $e(\c{H})\ge k n^\be$ and $k\ge C\ge (30)^{4h^2}$.  For all $1\le i\le h$, there exist subsets $V_i'\sub V_i$ and reals $0\le p_i\le 1$ such that the following hold with respect to $\c{H}'=\c{H}[\bigcup_{i=1}^h V_i']$ and $\c{F}'=\c{F}[\bigcup_{i=1}^h V_i']$:
	\begin{itemize}
		\item[(a)] $e(\c{H}')\ge  \prod_i p_i e(\c{H})$.
		\item[(b)] For every type $\mathbf{t}$ we have $e_{\mathbf{t}}(\c{F}')\le (40 h)^{4h^2f'} \prod_i p_i^{t_i} e_{\mathbf{t}}(\c{F})$.
		\item[(c)] For all $i$, either $|V'_i|=1$ or $|V'_i|\le (20)^{4h^2} p_i n$.
		\item[(d)] The number of $i$ with $|V_i'|=1$ is some number $0\le s< \lceil h-\be\rceil $.
		\item[(e)] If $|V_i'|\ge  2$, then 
		\[p_i=p:=\left[C^{-1} \prod_{i':|V_{i'}'|=1} p_{i'} k \right]^{\frac{-1}{h-\be-s}}.\]
		Moreover, we have $n^{-1}\le p\le (C^{-1} k)^{\frac{-1}{h-\be}}$ and $p_i\ge p$ whenever $|V_i'|=1$.
	\end{itemize}
\end{prop}

Note that the exponent in the definition of $p$ in (e) is negative and well-defined by (d). We also note that the hypothesis of the lemma implies $\be<h$, since $n$-vertex $h$-partite hypergraphs trivially have $e(\c{H})\le n^h$ and we assumed $e(\c{H}) \ge kn^\be$.

\begin{proof}
	Our goal will be to iteratively construct sets $V_1^{(j)},\ldots,V_{h}^{(j)}$ and reals $p_1^{(j)},\ldots,p_{h}^{(j)}$ which satisfy the following for all $j$ with respect to $\c{H}^{(j)}=\c{H}[\bigcup_i V_i^{(j)}]$ and $\c{F}^{(j)}=\c{F}[\bigcup_i V_i^{(j)}]$:
	\begin{itemize}
		\item[(a')] $e(\c{H}^{(j)})\ge \prod_i p_i^{(j)} e(\c{H})$.
		\item[(b')] For every type $\mathbf{t}$ we have $e_{\mathbf{t}}(\c{F}^{(j)})\le (40 h)^{jf'} \prod_i (p_i^{(j)})^{t_i} e_{\mathbf{t}}(\c{F})$.
		\item[(c')] For all $i$, either $|V_i^{(j)}|=1$ or $|V_i^{(j)}|\le (20)^{j} p_i^{(j)} n$.
		\item[(d')] The number of $i$ with $|V_i^{(j)}|=1$ is some number $0\le s^{(j)}<\lceil h-\be\rceil $.
		\item[(e')] For all $i$ we have 
		\[p_i^{(j)}\ge p^{(j)}:=\left[C^{-1} \prod_{i':|V_{i'}^{(j)}|=1} p_{i'}^{(j)} k \right]^{\frac{-1}{h-\be-s^{(j)}}},\]
		and moreover we have $n^{-1}\le p^{(j)}\le (C^{-1} k)^{-\frac{1}{h-\be}}$.
		\item[(f')] If $j\ge 1$ then $s^{(j)}\ge s^{(j-1)}$, and  if $s^{(j)}=s^{(j-1)}$ then
		\[|\{i: p_i^{(j)}=p^{(j)}\}|>|\{i:p_i^{(j-1)}=p^{(j-1)}\}|.\]
	\end{itemize}
	Observe that if we can simultaneously satisfy conditions (a') through (e') and additionally that $p_i^{(j)}=p^{(j)}$ whenever $|V_i^{(j)}|\ne 1$ for some $j\le 4h^2$, then the sets $V_i^{(j)}$ and reals $p_i^{(j)}$ will give the result.  The condition (f') will be used only to ensure that $j\le 4h^2$ at the end of this process.
	
	To begin, we set $V_i^{(0)}=V_i$ and $p_i^{(0)}=1$ for all $i$.
	\begin{claim}\label{start of procedure}
		The sets $V_i^{(0)}$ and reals $p_i^{(0)}$ satisfy conditions (a') through (f').
	\end{claim}
	\begin{proof}
		Conditions (a'), (b'), (c'), and (f') are trivially satisfied.  For ease of notation for the remaining conditions (d') and (e'), we let $s=s^{(0)}$ be the number of $i$ with $|V_i|=1$.  Since each edge of $\c{H}$ can be defined by picking a vertex of $V_i$, we trivially have
		\[k n^\be \le  e(\c{H})\le \prod_i |V_i|\le  n^{h-s},\]
		and hence
		\begin{equation}
			k\le n^{h-\be-s} \label{eq:kboundStart}.
		\end{equation}
		If $s\ge \lceil h-\be \rceil\ge h-\be$, then \eqref{eq:kboundStart} implies $k\le 1$, a contradiction to our assumption on $k$.  We conclude that $s< \lceil h-\be \rceil$, so (d') is satisfied.  Moving to (e'), we see that this condition is equivalent to having
		\[n^{-1}\le (C^{-1} k)^{\frac{-1}{h-\be-s}}\le (C^{-1} k)^{\frac{-1}{h-\be}}.\]
		The upper bound follows from $s\le \lceil h-\be\rceil-1<h-\be$ and $k\ge C$.  The lower bound follows from \eqref{eq:kboundStart} and $C\ge 1$.  We conclude that all of the conditions are satisfied.
	\end{proof}
	
	Now assume we have constructed sets $V_i^{(j-1)}$ and reals $p_i^{(j-1)}$ satisfying (a') through (f') for some $j\ge 1$.  If $p_i^{(j-1)}=p^{(j-1)}$ for all $i$ with $|V_i^{(j-1)}|\ge 2$ then we stop the process.  Otherwise, by (e') there must exist some $i$ with $p_i^{(j-1)}>p^{(j-1)}$ and with $|V_i^{(j-1)}|\ge 2$, and we fix such an $i$ arbitrarily.  For all $i'\ne i$, we set $V_{i'}^{(j)}=V_{i'}^{(j-1)}$ and $p_{i'}^{(j)}=p_{i'}^{(j-1)}$.  For $i'=i$, we apply \Cref{shrinkTechnical} to the pair $(\c{H}^{(j-1)},\c{F}^{(j-1)})$ with $p:=p^{(j-1)}/p_i^{(j-1)}$ and let $V_i^{(j)}\sub V_i^{(j-1)}$ denote the set guaranteed by this lemma and we let $p_i^{(j)}=p_i^{(j-1)}\cdot q$ where $q$ is the real number guaranteed from the lemma.
	\begin{claim}
		The sets $V_i^{(j)}$ and reals $p_i^{(j)}$ satisfy conditions (a') through (f') provided $j\le 4h^2$.
	\end{claim}
	\begin{proof}
		For (a'), we have by \Cref{shrinkTechnical}(a) and the fact that the $V_i^{(j-1)}$ sets and reals $p_i^{(j-1)}$ satisfied (a') that
		\begin{align*}e(\c{H}^{(j)})\ge q e(\c{H}^{(j-1)})&\ge q \prod_{i'} p_{i'}^{(j-1)} e(\c{H})\\ &= qp_i^{(j-1)}\prod_{i'\ne i} p_{i'}^{(j)} e(\c{H})=\prod_{i'} p_{i'}^{(j)} e(\c{H}).\end{align*}
		This proves (a'), and essentially identical arguments give (b') and (c') as well. For the remaining conditions we let \[s:=s^{(j)}.\]  We split our argument into two cases depending on the size of $V_i^{(j)}$.  First assume $|V_i^{(j)}|\ge 2$.  Since $V_{i'}^{(j)}=V_{i'}^{(j-1)}$ for $i'\ne i$, this means $s^{(j)}=s^{(j-1)}=s$, so (d') holds for $j$ since it held for $j-1$.  Similarly since $p_{i'}^{(j)}=p_{i'}^{(j-1)}$ for $i'\ne i$ and in particular for all $i'$ with $|V_{i'}^{(j)}|=1$, we have that $p^{(j)}=p^{(j-1)}$ by definition of this real number.  Since (e') held for $j-1$, we continue to have $n^{-1}\le p^{(j)}\le 1$ and $p_{i'}^{(j)}=p_{i'}^{(j-1)}\ge p^{(j)}$ for all $i'\ne i$.  For $i'=i$, because we are assuming the set $V_i^{(j)}$ given by \Cref{shrinkTechnical} has size at least 2, \Cref{shrinkTechnical}(c) guarantees $q=p=p^{(j-1)}/p_i^{(j-1)}$, and hence $p_i^{(j)}=p_i^{(j-1)} q=p^{(j-1)}=p^{(j)}$.  Thus (e') is satisfied, and moreover this equality $p_i^{(j)}=p^{(j)}$ and the assumption $p_i^{(j-1)}>p^{(j)}$ by our choice of $i$ implies (f') holds.  This establishes the result when $|V_i^{(j)}|\ge 2$.
		
		From now on we assume $|V_i^{(j)}|=1$, which means $s^{(j)}=s=s^{(j-1)}+1$ and hence (f') is trivially satisfied.  As in our proof of \Cref{start of procedure}, we note that because edges of $\c{H}^{(j)}$ can be identified by selecting a vertex from each $V_{i'}^{(j)}$ set, we have by us having  proven (a') and (c') hold for $j$ that
		\begin{equation}  \prod_{i'} p_{i'}^{(j)} \cdot k n^\be\le  e(\c{H}^{(j)}) \le \prod_{i'} |V_{i'}^{(j)}|\le \prod_{i': |V_{i'}^{(j)}|\ge 2} (20)^{j}p_{i'}^{(j)} n.\label{eq:lastStep}\end{equation}
		Dividing both sides by the $h-s$ common factors of $p_{i'}^{(j)}$ gives
		\[(20)^{j(h-s)} n^{h-s}\ge \prod_{i':|V_{i'}^{(j)}|=1} p_{i'}^{(j)}\cdot k n^\be= p_i^{(j)} \prod_{i' :|V_{i'}^{(j-1)}|=1} p_{i'}^{(j-1)}\cdot k n^\be =p_i^{(j)} C (p^{(j-1)})^{-h+\be+s-1} n^\be,\]
		with this last step using the definition of $p^{(j-1)}$ and that $s^{(j-1)}=s-1$.  Rearranging this further together with the definition of $p_i^{(j)}$ gives
		\[p_i^{(j-1)} q= p_i^{(j)}\le C^{-1} (20)^{j(h-s)} p^{(j-1)} (p^{(j-1)} n)^{h-\be-s}<  p^{(j-1)}(p^{(j-1)} n)^{h-\be-s},\]
		where this last inequality used $j\le 4h^2$ and $C>(20)^{4h^3}$.  Note that if $s\ge h-\be$ then this together with $p^{(j-1)}\ge n^{-1}$ from (e') for $j-1$ implies $q< p^{(j-1)}/p_i^{(j-1)}=p$, a contradiction to \Cref{shrinkTechnical}.  We thus must have $s<h-\be\le \lceil h-\be\rceil$, proving (d').
		
		For (e'), we first claim that $p^{(j)}\le p^{(j-1)}$.  Indeed, because $|V_i^{(j)}|=1$ the real number $p_i^{(j)}$ appears in the definition of $p^{(j)}$ with a negative exponent since by (d') we have $h-\be-s>0$.  As such, $p^{(j)}$ is as large as possible when $p_i^{(j)}$ is as small as possible, and for this we observe \[p_i^{(j)}=p_i^{(j-1)} q\ge p^{(j-1)}\] since $q\ge p=p^{(j-1)}/p_i^{(j-1)}$.  At this extremal value of $p_i^{(j)}=p^{(j-1)}$, the definition of $p^{(j)}$ implies
		\[(p^{(j)})^{-h+\be+s}=C^{-1} \prod_{i':|V_{i'}^{(j)}|=1} p_{i'}^{(j)} k=C^{-1} p^{(j-1)} \prod_{i':|V_{i'}^{(j-1)}|=1} p_{i'}^{(j-1)} k=p^{(j-1)}\cdot (p^{(j-1)})^{-h+\be+s-1},\]
		so in fact $p^{(j)}=p^{(j-1)}$ for this largest possible value of $p^{(j)}$, proving that $p^{(j)}\le p^{(j-1)}$.
		
		With the inequality $p^{(j)}\le p^{(j-1)}$ above in mind, we immediately have for $i'\ne i$ by (e') holding for $j-1$ that $p_{i'}^{(j)}=p_{i'}^{(j-1)}\ge p^{(j-1)}\ge p^{(j)}$, and by the analysis in the paragraph above we also have $p_i^{(j)}\ge p^{(j-1)}\ge p^{(j)}$.  Similarly, the inequality $p^{(j)}\le (C^{-1} k)^{\frac{-1}{h-\be}}$ follows from this holding for $j-1$, so all that remains to check is $p^{(j)}\ge n^{-1}$.  Again by using \eqref{eq:lastStep} and the definition of $p^{(j)}$ we find
		\[(20)^{h^3} n^{h-s}\ge  \prod_{i':|V_{i'}^{(j)}|=1} p_{i'}^{(j)}\cdot k n^\be= C (p^{(j)})^{-h+\be+s} n^\be.\]
		Rearranging and using that $h-\be-s>0$ and $C\ge (20)^{h^3}$ exactly implies $p^{(j)}\ge n^{-1}$, finishing the proof of (e') and the claim.
	\end{proof}
	With this claim we see that the sets $V_i^{(j)}$ and reals $p_i^{(j)}$ will continue to be defined and satisfy the stated conditions until either the process terminates (meaning $p_i^{(j)}=p^{(j)}$ for all $i$ with $|V_i^{(j)}|\ge 2$) or until $j=4h^2$.  We claim that in fact the process will terminate at some $j\le 4h^2$.  Indeed, by  (f'), there can exist at most $h+1$ indices with $s^{(j)}=s$ for a given value $s$.  Since there are trivially at most $h+1$ possible values of $s^{(j)}$, we conclude that there are at most $(h+1)^2\le 4h^2$ possible indices $j$ for which (f') can be satisfied, and since (f') is always satisfied if the process both does not terminate and if $j\le 4h^2$, we conclude that the process must terminate for some $j\le 4h^2$, proving the result.
\end{proof}

We use this to prove a more quantitative version of \Cref{supersaturation for f large} for partite $\c{H}$.  For this we adopt the shorthand $|\mathbf{t}|=\sum t_i$ for a type $\mathbf{t}$.

\begin{prop}\label{supersaturationCase1}
	For all positive integers $h,f'$ and reals $\be<h$ and $C'\ge 1$, there exist real numbers $k_0,c>0$ such that the following holds.  Let $\c{P}$ be a hereditary family of hypergraph pairs with $\ex(n,\c{P})\le C' n^\be$ for all $n$, and let $(\c{H},\c{F})\in \c{P}_n$ be such that $\c{H}$ has an $h$-partition $V_1,\ldots,V_h$, $e(\c{H})\ge k n^\be$ with $k\ge k_0$, and every edge of $\c{F}$ has size at most $f'$.  If either (i) every edge of $\c{F}$ has size at least $h-\be$ or (ii) if there does not exist any $(\c{H}',\c{F}')\in \c{P}$ and $e\in E(\c{F}'),e'\in E(\c{H}')$ with $e\sub e'$; then there exists a type $\mathbf{t}$ such that
	\[e_{\mathbf{t}}(\c{F})\ge \begin{cases}
		c k^{\frac{|\mathbf{t}|}{h-\be}} & |\mathbf{t}|\ge h-\be,\\ 
		c k^{\frac{2}{h-\be-|\mathbf{t}|+2}} & |\mathbf{t}|\le h-\be.
	\end{cases}\] 
\end{prop}
\begin{proof}
	Let $k_0,c$ be real numbers to be specified later and assume for contradiction that there is an $(\c{H},\c{F})\in \c{P}_n$ satisfying either (i) or (ii) for which the proposition fails to hold.  Let $\c{H}',\c{F}'$ be the induced subhypergraphs of $\c{H},\c{F}$ guaranteed by \Cref{shrinkingApplied} with $C=\max\{C' h^\be (30)^{4h^2 \be},(30)^{4h^2}\}$, where here we implicitly assume $k\ge k_0\ge C$ so that this proposition applies.
	\begin{claim}
		It suffices to show $e(\c{F}')=0$.
	\end{claim}
	\begin{proof}
		By \Cref{shrinkingApplied}(c) and (e), we have $|V_i'|\le (20)^{4h^2} p n$ for all $i$, and hence \[m:=v(\c{H}')\le h (20)^{4h^2} pn.\]  
		
		Observe that $(\c{H}',\c{F}')\in \c{P}_m$ because $\c{P}$ is hereditary, and our assumption $e(\c{F}')=0$ implies $\al_{\c{H}'}(\c{F}')=e(\c{H}')$.  This together with our hypothesis on $\ex(m,\c{P})$ gives
		\[C' h^\be (20)^{4h^2\be} p^\be n^\be\ge \ex(m,\c{P})\ge e(\c{H}')\ge  \prod_i p_i k n^\be=p^{h-s}\cdot C p^{-h+\be+s} n^\be,\]
		where here the last inequality used \Cref{shrinkingApplied}(a) and the last equality used that $p_i=p$ for the $h-s$ indices with $|V_i'|\ge 2$ together with the definition of $p$.  Plugging in $C\ge C'h^\be (30)^{4h^2\be }$ gives a contradiction.
	\end{proof}
	This in turn implies the following.
	\begin{claim}\label{type bound claim}
		For each type $\mathbf{t}$ let \[I_{\mathbf{t}}=\{i:|V'_i|=1,\ t_i=1\}.\]
		It suffices to show for all types $\mathbf{t}$ with $e_{\mathbf{t}}(\c{F})>0$  and  $t_i\le 1$ for $i$ with $|V_i'|=1$ that
		\begin{equation}p^{|\mathbf{t}|-|I_{\mathbf{t}}|} \prod_{i\in I_{\mathbf{t}}} p_i<(40 h)^{-4h^2f'} c^{-1} e_{\mathbf{t}}(\c{F})^{-1}.\label{eq:target}\end{equation}
	\end{claim}
	\begin{proof}
		By the previous claim it suffices  to show $e(\c{F}')=0$ which is equivalent to showing $e_{\mathbf{t}}(\c{F}')=0$ for all types $\mathbf{t}$.  This is trivially true if $e_{\mathbf{t}}(\c{F})=0$ or if $t_i>1$ for some $i$ with $|V_i'|=1$, so we need only consider types as in the claim. Note that by \Cref{shrinkingApplied}(b), for any type $\mathbf{t}$ we have
		\[e_{\mathbf{t}}(\c{F}')\le (40 h)^{4h^2f'}\prod_i p_i^{t_i} e_{\mathbf{t}}(\c{F}),\]
		so to show $e_{\mathbf{t}}(\c{F}')<1$ it suffices to show
		\[\prod_i p_i^{t_i}< (40 h)^{-4h^2f'}c^{-1} e_{\mathbf{t}}(\c{F})^{-1}.\]
		Moreover, we have $\prod p_i^{t_i}=p^{|\mathbf{t}|-|I_{\mathbf{t}}|} \prod_{i\in I_{\mathbf{t}}} p_i$ since $p_i=p$ whenever $|V'_i|\ge 2$ and $t_i\le 1$ whenever $|V'_i|=1$ by assumption, which together with the inequality above gives \eqref{eq:target}.
	\end{proof}
	
	Let $\mathbf{t}$ be any type as in \Cref{type bound claim} and first consider the case that $|\mathbf{t}|\ge h-\be$.  By using the inequality $p\le p_i$ from \Cref{shrinkingApplied}(e) for each of the $s-|I_{\mathbf{t}}|$ indices $i$ with $|V'_i|=1$ and $i\notin I_{\mathbf{t}}$, we find that
	\[p^{|\mathbf{t}|-|I_{\mathbf{t}}|} \prod_{i\in I_{\mathbf{t}}} p_i\le p^{|\mathbf{t}|-s} \cdot \prod_{i:|V'_i|=1} p_i=p^{|\mathbf{t}|-s}\cdot C k^{-1} p^{-h+\be+s}=p^{|\mathbf{t}|-h+\be}  Ck^{-1}\le (C k^{-1})^{\frac{|\mathbf{t}|}{h-\be}},\]
	where this last inequality used $|\mathbf{t}|\ge h-\be$ and $p\le (C^{-1} k)^{-\frac{1}{h-\be}}$ from \Cref{shrinkingApplied}(e).  Recall now at the start of the proof that we implicitly assumed for contradiction (regardless of if we are in case (i) or (ii)) that every $\mathbf{t}$ with $|\mathbf{t}|\ge h-\be$ has $e_{\mathbf{t}}(\c{F})\le c k^{\frac{|\mathbf{t}|}{h-\be}}$, and combining this with the inequality above yields \eqref{eq:target} provided $c$ is sufficiently small in terms of $C^{\frac{|\mathbf{t}|}{h-\be}}\le C^{\frac{f'}{h-\be}}$, with this last inequality using that $|\mathbf{t}|\le f'$ for any type with $e_{\mathbf{t}}(\c{F})>0$.

	From now on we consider the case $|\mathbf{t}|< h-\be$.  Because we are assuming $e_{\mathbf{t}}(\c{F})>0$, this means $\c{F}$ contains an edge of size $|\mathbf{t}|<h-\be$ which means we are in case (ii). Observe from $p_i\ge p$ that
	\[p^{-h+\be+s}= C^{-1}k\prod_{i:|V_i'|=1} p_i \ge C^{-1} k\prod_{i\in I_{\mathbf{t}}} p_i\cdot p^{s-|I_{\mathbf{t}}|},\]
	and rearranging this together with the fact that $h-\be-|I_{\mathbf{t}}|\ge h-\be-s>0$ gives
	\[p \le \left(C^{-1} k\prod_{i\in I_{\mathbf{t}}} p_i\right)^{-\frac{1}{h-\be-|I_{\mathbf{t}}|}}.\]
	Applying this bound gives
	\begin{equation}p^{|\mathbf{t}|-|I_{\mathbf{t}}|} \prod_{i\in I_{\mathbf{t}}} p_i\le (C^{-1} k)^{\frac{|\mathbf{t}|-|I_{\mathbf{t}}|}{h-\be-|I_{\mathbf{t}}|}}\left(\prod_{i\in I_{\mathbf{t}}} p_i\right)^{\frac{h-\be-|\mathbf{t}|}{h-\be-|I_{\mathbf{t}}|}}\le (C^{-1} k)^{\frac{|\mathbf{t}|-|I_{\mathbf{t}}|}{h-\be-|I_{\mathbf{t}}|}},\label{eq:It}\end{equation}
	where this last step used $p_i\le 1$ for all $i$ and that the exponent of $\prod p_i$ is non-negative since we assumed $|\mathbf{t}|<h-\be$ in this case.  
	
	Because $|\mathbf{t}|<h-\be$, the upper bound of \eqref{eq:It} is as large as possible when $|I_{\mathbf{t}}|$ is as large as possible.  Crucially, because no edge of $\c{F}$ is contained in an edge of $\c{H}$ by hypothesis of (ii), any type $\mathbf{t}$ with $e_{\mathbf{t}}(\c{F})>0$ can not have $|I_{\mathbf{t}}|=|\mathbf{t}|$, for if it did then any edge $e\in E(\c{F})$ of this type and any edge $e'\in E(\c{H}')$ (which exists by \Cref{shrinkingApplied}(a) since $e(\c{H})>0$) would necessarily have $e\sub e'$ since $e'$ uses one vertex from each $V'_i$, and in particular it uses all $|I_{\mathbf{t}}|=|\mathbf{t}|$ vertices of $e$.  Also observe that if $|I_{\mathbf{t}}|<|\mathbf{t}|$ then $|I_{\mathbf{t}}|\le |\mathbf{t}|-2$ since such a $\mathbf{t}$ has some $t_i\ge 2$ and hence $|I_{\mathbf{t}}|\le |\mathbf{t}|-t_i\le |\mathbf{t}|-2$. Plugging in this estimate $|I_{\mathbf{t}}|\le |\mathbf{t}|-2$ into \eqref{eq:It} gives \eqref{eq:target} under our hypothesis $e_{\mathbf{t}}(\c{F})\le c k^{\frac{2}{h-\be-|\mathbf{t}|+2}}$ provided  $c$ is sufficiently small, completing the proof.
\end{proof}

We now formally put these pieces together to prove our main supersaturation result.
\begin{proof}[Proof of \Cref{supersaturation for f large}]
	Let $\c{P}$ be a hereditary family of $(h,f,f')$-bounded hypergraph pairs such that either $f\ge h-\be$ or no edge of any $\c{H}$ contains an edge of any $\c{F}$, and let $\be<h$  and $C'\ge 1$ be real numbers such that $\ex(n,\c{P})\le C' n^\be$.  Let $k_0(\be,C',\hat{h},f')$ denote the constant from \Cref{supersaturationCase1} applied with $\hat{h}$ in place of $h$, similarly define $c(\be,C',\hat{h},f)$, and let $k_0=2 h^{h+1} (1+\max_{\be < \hat{h}\le h} k_0(\be,C',\hat{h},f'))$ and $c=\min_{\be< \hat{h}\le h} c(\be,C',\hat{h},f)$.  We aim to prove the theorem with respect to this constant $k_0$. To aid with this proof, we define a new family $\c{P}'=\{(\c{H}',\c{F}):(\c{H},\c{F})\in \c{P},\ \c{H}'\sub \c{H}\}$.  It is not difficult to see that $\c{P}'$ is also a hereditary family of $(h,f,f')$-bounded hypergraph pairs and that $\ex(n,\c{P}')=\ex(n,\c{P})$ since $\al_{\c{H}'}(\c{F})\le \al_{\c{H}}(\c{F})$ for $\c{H}'\sub \c{H}$.
	
	Consider any $(\c{H},\c{F})\in \c{P}$ such that $e(\c{H})\ge k n^\be$ with $k\ge k_0$.  By our definition of $k_0$ we have $e(\c{H})\ge 2$, so by \Cref{partite} there exists some $\hat{\c{H}}\sub \c{H}$ on the same vertex set of $\c{H}$ which has a $\hat{h}$-partition $V_1,\ldots,V_{\hat{h}}$ and \[e(\hat{\c{H}})\ge \half h^{-h-1} k n^\be\ge \max_{\be<\hat{h}'\le h}k_0(\be,C',\hat{h}',f')n^\be.\] Observe that we must have $\hat{h}>\be$ in order for the $n$-vertex $\hat{h}$-uniform hypergraph $\hat{\c{H}}$ to have at least this many edges, so in particular $e(\hat{\c{H}})\ge k_0(\be,C',\hat{h},f')n^\be$.  Note that $(\hat{\c{H}},\c{F})\in \c{P}'$ by definition of this family, so it follows from \Cref{supersaturationCase1} applied to $\c{P}'$ (which we can apply since either $f\ge h-\be$ or no edge of any $\c{H}$ contains an edge of any $\c{F}$) that there is some type $\mathbf{t}$ such that $e_{\mathbf{t}}(\c{F})$ is as large as stated in \Cref{supersaturationCase1}.  In particular $e_{\mathbf{t}}(\c{F})>0$, which by assumption of $(\c{H},\c{F})$ being $(h,f,f')$-bounded means $f\le |\mathbf{t}|\le f'$.
	
	First consider the case $f\ge h-\be$.  This means $|\mathbf{t}|\ge h-\be$, and hence by \Cref{supersaturationCase1},
	\[e(\c{F})\ge e_{\mathbf{t}}(\c{F})=\Om(k^{\frac{|\mathbf{t}|}{h-\be}})=\Om(k^{\frac{f}{h-\be}}),\]
	proving the result. 	Now consider the case $f<h-\be$. This implies
	\[\frac{|\mathbf{t}|}{h-\be}\ge \frac{f}{h-\be}\ge \frac{f-(f-2)}{h-\be-(f-2)}=\frac{2}{h-\be-f+2}\ge \frac{2}{h-\be-|\mathbf{t}|+2}.\]
	As such, if $|\mathbf{t}|\ge h-\be$ then the desired bound holds from \Cref{supersaturationCase1}, so we may assume $|\mathbf{t}|<h-\be$.  This in particular means $|\mathbf{t}|\le \min\{f',\lceil h-\be\rceil -1\}$, and hence
	\[e(\c{F})\ge e_{\mathbf{t}}(\c{F})=\Om(k^{\frac{2}{h-\be-|\mathbf{t}|+2}})=\Om(k^{\frac{2}{h-\be-\min\{f',\lceil h-\be\rceil-1\}+2}}),\]
	proving the result.
\end{proof}

\section{Concluding Remarks}\label{sec:concluding}
In this paper we developed a way to view extremal problems as hypergraph-weighted independent sets, proved general supersaturation results for these extremal functions, and used this to give supersaturation results for generalized Tur\'an problems and related problems on linear equations.  We close by briefly discussing possible improvements of our results as well several new problems motivated by our hypergraph-weighted framework.

\textbf{Improved Bounds.}  The first bound of \Cref{supersaturation for f large} is a natural barrier for our methods and can be sharp in certain cases.  The second bound of \Cref{supersaturation for f large} is not natural, and in particular can be improved whenever a stronger condition on $(\c{H},\c{F})$ is imposed than just that no edge of $\c{H}$ contains an edge of $\c{F}$.  Specifically, if we further impose that for every $(\c{H},\c{F})\in \c{P}$ and every non-empty partite subgraph $\c{H}'\sub \c{H}$ that each edge of $\c{F}$ intersects at most $r<f$ of these parts in exactly 1 vertex, then we can obtain a supersaturation bound of the form $e(\c{F})=\Om(k^{\frac{f-r}{h-\be-r}})$ by our same proof since in \Cref{supersaturationCase1} we can by definition impose $|I_{\mathbf{t}}|\le r$.  Such a bound can be useful, for example, in generalized Tur\'an problems for certain choices of graphs $H,F$.  However, we wonder if this bound can simply always be improved to the more natural barrier in the first case of the theorem.

\begin{quest}
	Can one replace the second bound of \Cref{supersaturation for f large} by $\Om(k^{\frac{f}{h-\be}})$?
\end{quest} 

While we stated many general applications of our supersaturation results, one can often get sharper bounds by using more specific features of the problem at hand.  For example, we did not state the application of \Cref{FMS bound} to generalized planar Tur\'an numbers because the only choice of $d$ which works universally is $d=h-2$ but such a bound is usually ineffective since the maximum number of copies of $H$ in an $n$ vertex planar graph is often smaller than $n^{h-2}$.  Similar situations can happen in the setting of systems of linear equations, for example if $H$ encodes $h$-AP's with $h\ge 2$ then one can take $d=1$ instead of $d=h-1$ like \Cref{systems} implicitly assumes. For situations such as these it is best to look at \Cref{FMS bound} and see what the best value of $d$ is that one can take for one's problem.

\textbf{New Problems.}  There are many other problems which can be viewed from the perspective of determining $\al_{\c{H}}(\c{F})$ which may be of interest.  For example, we can ask the following variant of $\ex_\N(n,H,F)$ where we look for sets $A$ contained in $[n]$ rather than of size exactly $n$.

\begin{prob}\label{equation problem}
	Given sets $H,F\sub 2^{\N}$ of sets of integers and an integer $n$, determine the maximum number of subsets of $H$ that a set $A\sub [n]$ can contain if $A$ does not contain any subset of $F$.
\end{prob}
When $H$ is the set of all 1-element subsets and $F$ is the set of all non-trivial solutions to a system of linear equations then this is exactly the framework developed by Ruzsa~\cite{ruzsa1993solving}.  One can view \Cref{equation problem} through our weighted independent set language by defining $\c{H}_n,\c{F}_n$ to be the hypergraphs on $[n]$ which contains $H\cap 2^{[n]}$ and $F\cap 2^{[n]}$ as their respective edges, in which case \Cref{equation problem} precisely asks for $\al_{\c{H}_n}(\c{F}_n)$.  It is worth noting that this family of pairs $(\c{H}_n,\c{F}_n)$ is not hereditary, and as such our results do not say anything about its supersaturation properties.

One can consider variants of $\al_{\c{H}}(\c{F})$ to obtain other problems, such as Ramsey type results.  For example, we recently introduced the notion of the $H$-Ramsey number of a graph $F$, which is the minimum number of copies of $H$ that a graph $G$ can have if every 2-coloring of the edges of $G$ contains a copy of $F$.  This definition recovers the usual Ramsey number when $H=K_1$ and the size Ramsey number when $H=K_2$; see \cite{fox2026triangle,luo2026off} for more on this.  To phrase this in terms of hypergraphs, we let $\c{H}_G,\c{F}_G$ denote the hypergraphs on $E(G)$ which encode copies of $H,F$, and we observe that the $H$-Ramsey number of $F$ is exactly the minimum value of $e(\c{H}_G)$ amongst all graphs $G$ such that $E(G)=V(\c{F}_G)$ can not be written as the union of 2 independent sets of $\c{F}_G$.  This framing of Ramsey problems in terms of independent sets of pairs of hypergraphs suggests a number of new problems, such as the following size-Ramsey analog of Van der Waerden's Theorem which appears to be unexplored.
\begin{prob}
	Determine the minimum number of $\ell$-AP's that a set $A\sub \N$ can have if $A$ has the property that any $r$-coloring of its elements contains a monochromatic $k$-AP.
\end{prob}
Analogous to Folkman's Theorem for graphs, it is known that the answer to this problem is 0 for $k>\ell$ due to independent work of Spencer~\cite{spencer1975restricted} and Ne\v{s}et\v{r}il and R\"odl~\cite{nevsetvril1976van}.

\begingroup
\footnotesize  
\bibliographystyle{amsplain}
\bibliography{bib}{}
\endgroup

\end{document}